\documentclass{amsart}
 
\usepackage{amssymb,latexsym}

\theoremstyle{plain}
\newtheorem{theorem}{Theorem}
\newtheorem{cor}{Corollary}
\newtheorem*{main1}{Transitivity Lemma}
\newtheorem{lemma}{Lemma}
\newtheorem{proposition}{Proposition}
\newtheorem{fact}{Fact}

\title[Galois Connections for External Operations and Constraints]
{On Galois Connections between External Operations and Relational Constraints: Arity Restrictions and 
Operator Decompositions}
\author[Miguel Couceiro]{Miguel Couceiro}

\date{February, 2005}
\address{Department of Mathematics, Statistics and Philosophy\\
University of Tampere\\ 
Kalevantie 4, 33014 Tampere, Finland}
\email{Miguel.Couceiro@uta.fi}
\thanks{The author wishes to thank Stephan Foldes for useful comments and remarks
 which contributed to improve the present manuscript.}
\thanks{Partially supported by the Graduate School in Mathematical Logic MALJA.
Supported in part by grant $\#$28139 from the Academy of Finland}

\begin{document}

\begin{abstract} We study the basic Galois connection  
induced by the "satisfaction" relation between external 
operations $A^n\rightarrow B$ defined on a set $A$
 and valued in a possibly different set $B$ on the one hand, and ordered pairs $(R,S)$ of relations 
$R\subseteq A^m$ and $S\subseteq B^m$, called relational constraints, on the other hand.
We decompose the closure maps associated with this Galois
connection, in terms of closure operators corresponding to simple closure conditions describing 
the corresponding Galois closed sets of functions and constraints.
 We consider further Galois correspondences by
restricting the sets of primal and dual objects to fixed arities. We describe the restricted 
Galois closure systems by means of parametrized analogues of the simpler closure conditions, and
present factorizations of the corresponding
Galois closure maps, similar to those provided in the unrestricted case.
\end{abstract}

\maketitle

\section{Introduction}

In this paper we analyse the basic Galois connection implicit in [CF1] which extends to the infinite case
the framework of Pippenger in [Pi2], where classes of external operations 
(i.e. functions defined on a set $A$
 and valued in a possibly different set $B$) are defined by the ordered pairs of 
relations, called \emph{relational constraints}, which they satisfy, and dually where sets of constraints
are characterized by the functions satisfying them. 
As presented in [CF1], the results in this bi-sorted
framework specialize to those concerning the fundamental Galois correspondence ${\bf Pol}-{\bf Inv}$ 
 between operations and relations (for finite underlying sets, see [BKKR,G,PK],
 and [Sz,P\" o1,P\" o2], for arbitrary sets).
In analogy with the universal algebra setting, we consider further Galois connections arising from the 
restriction of the sets of functions and constraints to fixed arities (see e.g. [P\" o1] and [P\" o2]).

In Section 2, we recall basic concepts and terminology, and introduce the fundamental Galois connection 
between external operations (functions) and relational constraints. 
The Galois closed sets with respect to this correspodence
are described in Section 3 by means of simple closure conditions provided in [CF1]. Also we define operators associated
with these conditions, and present factorizations of the closure maps associated with this Galois connection,
 analogous to those given in [P\" o2].
In Section 4, we study further Galois correspondences induced by the
 restiction of the sets of primal and dual objects
 to fixed arities. To characterize the corresponding Galois closed sets of functions and constraints,
 we define
parametrized analogues of the simple conditions and corresponding closure operators, given in Section 3, and represent the
restricted Galois closure maps as compositions of these simpler closure operators.

\section{Basic Notions and Terminology}

Let $A$, $B$ and $E$ be arbitrary non-empty sets. 
 A \emph{$B$-valued function on $A$} (or, \emph{external operation})
is a map $f: A^n \rightarrow B$,
 for some positive integer $n$ called the \emph{arity} of $f$. For each positive integer $n$, we denote by 
$\bf n$ the set ${\bf n}=\{1,\ldots ,n\}$, so that the $n$-tuples ${\bf a}=(a_1,\ldots ,a_n)\in A^n$ can
 be thought of as unary $A$-valued functions ${\bf a}: {\bf n}\rightarrow A$ on $\bf n$ defined by 
${\bf a}(i)=a_i$.  
A \emph{class} of $B$-valued functions on $A$ is a subset $\mathcal{F}\subseteq \cup _{n\geq 1}B^{A^n}$.
For $A=B$, $A$-valued functions on $A$ are usually called
(\emph{internal}) \emph{operations on $A$}.
 For each positive integer $n$, the $n$-ary operations $(a_t\mid t\in {\bf n})\mapsto a_i$, 
$i\in {\bf n}$, 
are called \emph{projections}.  
 The \emph{composition} of an $n$-ary $E$-valued function $f$ on $B$ with $m$-ary $B$-valued functions 
$g_1,\ldots ,g_n$ on $A$ is the $m$-ary $E$-valued function $f(g_1,\ldots ,g_n)$ on $A$, defined by  
\begin{displaymath}
f(g_1,\ldots ,g_n)({\bf a})=f(g_1({\bf a}),\ldots ,g_n({\bf a})) 
\end{displaymath}
for every ${\bf a}\in A^m$. Composition is naturally extended to classes of functions.
 For $\mathcal{I}\subseteq \cup _{n\geq 1}E^{B^n}$ and $\mathcal{J}\subseteq \cup _{n\geq 1}B^{A^n}$,
 the \emph{composition of} $\mathcal{I}$ \emph{with} $\mathcal{J}$,
 denoted $\mathcal{I}\mathcal{J}$, is defined by
 \begin{displaymath}
 \mathcal{I}\mathcal{J}=\{f(g_1,\ldots ,g_n)\mid n,m\geq 1, f\textrm{ $n$-ary in $\mathcal{I}$, }g_1,\ldots ,g_n\textrm{ $m$-ary in $\mathcal{J}$} \}. 
\end{displaymath}
Note that for arbitrary non-empty sets $A$, $B$, $E$ and $G$, and function classes
$\mathcal{I}\subseteq \cup _{n\geq 1}G^{E^n}$, $\mathcal{J}\subseteq \cup _{n\geq 1}E^{B^n}$, and
$\mathcal{K}\subseteq \cup _{n\geq 1}B^{A^n}$, we have 
 $(\mathcal{I}\mathcal{J})\mathcal{K}\subseteq \mathcal{I}(\mathcal{J}\mathcal{K})$.
 (For background on class composition see [CF2], [CF3], and [CFL] in the Boolean case $A=B=\{0,1\}$.)

A \emph{clone} on $A$ is a class $\mathcal{C}\subseteq \cup _{n\geq 1}A^{A^n}$ of operations on $A$ 
 containing all projections, and satisfying $\mathcal{C} \mathcal{C}=\mathcal{C}$. 
We denote by $\mathcal{O}_A$ the smallest clone on $A$ containing only projection maps.
  
 \smallskip

For a positive integer $m$, an $m$\emph{-ary relation on $A$} is a subset $R$ of $A^m$, i.e
 a class of unary $A$-valued functions ${\bf a}: {\bf m}\rightarrow A$ defined on $\bf m$.
We use $=_A$ to denote the binary equality relation on a set $A$. 
 For an $n$-ary function $f\in B^{A^n}$ we denote by $fR$ the class composition
 \begin{displaymath}
 \{f\}R= \{f({\bf a}^1\ldots {\bf a}^n)\mid {\bf a}^1,\ldots ,{\bf a}^n \in R\}
\end{displaymath} 
In the particular case $A=B$, if $fR\subseteq R$, then $f$ is said to \emph{preserve} $R$.

\smallskip

An $m$-ary \emph{$A$-to-$B$ relational constraint} (or simply, $m$-ary \emph{constraint})
  is an ordered pair $(R,S)$ where $R\subseteq A^m$ and $S\subseteq B^m$ are called 
the \emph{antecedent} and \emph{consequent}, respectively,
 of the constraint (see [Pi2] and [CF1]).
 A $B$-valued function on $A$, 
 $f:A^n\rightarrow B$, $n\geq 1$, is said to \emph{satisfy} an $m$-ary $A$-to-$B$ constraint $(R,S)$ 
if $fR\subseteq S$. In other words, the function $f:A^n\rightarrow B$ satisfies the constraint $(R,S)$
if and only if $f$ is a homomorphism from the relational structure 
${\bf A}^n= \langle A^n,R^n \rangle $ to the relational structure  
${\bf B}= \langle B,S \rangle $. 
Note that every $B$-valued function on $A$ satisfies the 
\emph{binary $A$-to-$B$ equality constraint} $(=_A,=_B)$, the \emph{empty constraint}
 $(\emptyset ,\emptyset )$, and, for each $m\geq 1$, the \emph{trivial constraint} $(A^m,B^m)$. 

For a set $\mathcal{T}$ of $A$-to-$B$ constraints, we denote by ${\bf FSC}(\mathcal{T})$ 
 the class of all $B$-valued functions on $A$ satisfying every member of $\mathcal{T}$. Dually, for a class $\mathcal{K}$ 
of $B$-valued functions on $A$, we denote by ${\bf CSF}( \mathcal{K})$ the
set of all $A$-to-$B$ constraints satisfied by every member of $\mathcal{K}$.
The notation ${\bf FSC}$ stands for ``functions satisfying constraints", 
while ${\bf CSF}$ stands for ``constraints satisfied by functions"
Consider the mappings ${\bf FSC}:\mathcal{T}\mapsto {\bf FSC}(\mathcal{T})$ and  
${\bf CSF}:\mathcal{K}\mapsto {\bf CSF}(\mathcal{K})$.    
By definition it follows that 
\begin{itemize}
\item[$(i)$] ${\bf FSC}$ and ${\bf CSF}$ are order reversing, i.e. if $\mathcal{T}\subseteq \mathcal{T}'$
 and $\mathcal{K}\subseteq \mathcal{K}'$,
 then ${\bf FSC}(\mathcal{T}')\subseteq {\bf FSC}(\mathcal{T})$) and
 ${\bf CSF}(\mathcal{K}')\subseteq {\bf CSF}(\mathcal{K})$, and 
 \item[$(ii)$] the compositions ${\bf FSC}\circ {\bf CSF}$ and ${\bf CSF}\circ {\bf FSC}$
 are extensive maps, i.e.
$ \mathcal{K}\subseteq {\bf FSC}({\bf CSF}(\mathcal{K})) $ and 
$\mathcal{T}\subseteq {\bf CSF}({\bf FSC}( \mathcal{T}))$.
\end{itemize}
Thus, the pair ${\bf FSC}-{\bf CSF}$ constitutes a Galois connection between external functions
 and relational constraints, and as a consequence we have
\begin{itemize}
\item[$(a)$] ${\bf FSC}\circ {\bf CSF}\circ {\bf FSC}={\bf FSC}$ and 
${\bf CSF}\circ {\bf FSC}\circ {\bf CSF}={\bf CSF}$, and 
\item[$(b)$] ${\bf FSC}\circ {\bf CSF}$  and ${\bf CSF}\circ {\bf FSC}$ are closure operators, i.e.
extensive, monotone and idempotent. 
\end{itemize}
The function classes and the sets of constraints fixed by the operators in $(b)$ are said to be
(\emph{Galois}) \emph{closed}.
 (For background on Galois connections, see e.g. [O] and [Pi1].)

\section{The Galois Connection ${\bf FSC}-{\bf CSF}$ }

In this section we recall basic theory in [CF1], and develop some factorization results 
for the composites ${\bf FSC}\circ {\bf CSF}$ and ${\bf CSF}\circ {\bf FSC}$.

 \smallskip

A class $\mathcal{K}\subseteq \cup _{n\geq 1}B^{A^n}$ of $B$-valued functions on $A$ 
 is said to be \emph{definable} (or \emph{defined}) by a set $\mathcal{T}$ of $A$-to-$B$ constraints,
 if $\mathcal{K}={\bf FSC}(\mathcal{T})$. Dually, a set $\mathcal{T}$ of $A$-to-$B$ constraints
is said to be \emph{characterized} by a set $\mathcal{K}$ of $B$-valued functions on $A$ 
if $\mathcal{T}={\bf CSF}(\mathcal{K})$. 
Thus the closed sets of functions and the closed sets of relational constraints
with respect to the Galois connection ${\bf FSC}-{\bf CSF}$ are precisely 
the classes of functions definable by constraints, and the sets of constraints characterized by functions. 
 
 \smallskip

In the case of finite underlying sets $A$ and $B$, Pippenger determined, in [Pi2], that the necessary
and sufficient conditions for a class of functions to be definable by a set of relational constraints
  are essentially closure under certain functional compositions.
An $m$-ary $B$-valued function $g$ on $A$ 
is said to be obtained from an $n$-ary $B$-valued function $f$ on $A$
 by \emph{simple variable substitution},
if there are $m$-ary projections $p_1,\ldots ,p_n\in \mathcal{O}_A$ such that 
$g=f(p_1,\ldots ,p_n)$. 
 A class $\mathcal{K}$ of $B$-valued functions on $A$ is said to be 
\emph{closed under simple variable substitutions}
if each function obtained from a function $f$ in $\mathcal{K}$ by simple variable substitution
 is also in $\mathcal{K}$, i.e. if $\mathcal{K}=\mathcal{K}\mathcal{O}_A$, where $\mathcal{O}_A$ 
denotes the smallest clone on $A$ containing only projections.
 For a class $\mathcal{K}$ of $B$-valued functions on $A$,
 we define the closure ${\bf VS}(\mathcal{K})$
of $\mathcal{K}$ under ``variable substitutions" by ${\bf VS}(\mathcal{K})=\mathcal{K}\mathcal{O}_A$.
This is indeed the smallest class containing $\mathcal{K}$ and closed under simple variable substitutions. 
Clearly, the map $\mathcal{K}\mapsto {\bf VS}(\mathcal{K})$ is extensive and monotone,
 and for any class $\mathcal{K}$, we have  
 \begin{displaymath}
{\bf VS}({\bf VS}(\mathcal{K}))= (\mathcal{K} \mathcal{O}_A)\mathcal{O}_A\subseteq \mathcal{K} (\mathcal{O}_A\mathcal{O}_A)=
\mathcal{K} \mathcal{O}_A={\bf VS}(\mathcal{K}), 
\end{displaymath}
 i.e. $\mathcal{K}\mapsto {\bf VS}(\mathcal{K})$ is also idempotent.

\begin{fact} The operator $\mathcal{K}\mapsto {\bf VS}(\mathcal{K})$ is a closure operator on 
$\cup _{n\geq 1}B^{A^n}$.
\end{fact}
 
As shown in [CF1], in the general case of arbitrary underlying sets $A$ and $B$,
 the above closure does not suffice to guarantee
function class definability by relational constraints; "local closure" is required on 
the class of functions. A class $\mathcal{K}\subseteq \cup _{n\geq 1}B^{A^n}$ is said to be 
\emph{locally closed} if it contains every function for which every restriction to a finite subset of its
 domain $A^n$ coincides with a restriction of some member of $\mathcal{K}$.
 For background on the analogous concept defined on sets of operations, see e.g. [G,P\" o1,P\" o2]. 
For any class of functions $\mathcal{K}\subseteq \cup _{n\geq 1}B^{A^n}$
 we denote by ${\bf Lo}(\mathcal{K})$
 the smallest locally closed class of functions containing $\mathcal{K}$, called the \emph{local closure}
of $\mathcal{K}$. In other words, ${\bf Lo}(\mathcal{K})$ is the class of functions obtained from 
$\mathcal{K}$ by adding all those functions whose restriction to each finite subset of its
 domain $A^n$ coincides with a restriction of some member of $\mathcal{K}$.

\begin{fact} The operator $\mathcal{K}\mapsto {\bf Lo}(\mathcal{K})$ is a closure operator on 
$\cup _{n\geq 1}B^{A^n}$.
\end{fact} Note that, if $A$ is finite, then ${\bf Lo}(\mathcal{K})=\mathcal{K}$ for every 
class $\mathcal{K}\subseteq \cup _{n\geq 1}B^{A^n}$, i.e. every class $\mathcal{K}$ is locally closed. 

\begin{theorem}\emph{([CF1]:)} 
Consider arbitrary non-empty sets $A$ and $B$.
A class $\mathcal{K}$ of $B$-valued functions on $A$ is definable by some set of
 $A$-to-$B$ constraints if and only if 
 $\mathcal{K}$ is locally closed and it is closed under simple variable substitutions.
\end{theorem}
In other words, the closed sets of functions for the Galois connection ${\bf FSC}-{\bf CSF}$
(i.e. of the form ${\bf FSC}(\mathcal{T})$ for some set $\mathcal{T}$ of relational constraints)
 are exactly those locally closed classes which are closed under simple variable substitutions.
In order to provide the characterization of the closed systems of the dual objects,
 i.e. relational constraints, we recall the following concepts introduced in [CF1].

 \bigskip

Let $A,B,C$ and $D$ be arbitrary sets.  
For any maps $f:A\rightarrow B$ and $g:C\rightarrow D$, the \emph{concatenation} of $f$ and $g$, 
denoted $gf$, is defined to be the map with domain $f^{-1}[B\cap C]$ and codomain $D$ given by 
$(gf)(a)=g(f(a))$ for all $a\in f^{-1}[B\cap C]$. Note that concatenation is associative. 

 \smallskip

Given a non-empty family $(g_i)_{i\in I}$ of maps, $g_i:A_i\rightarrow B_i$ 
where $(A_i)_{ i\in I}$ is a family of pairwise disjoint sets,
 we denote by ${\Sigma }_{i\in I}g_i$, the map from ${\cup }_{i\in I}A_i$ to ${\cup }_{i\in I}B_i$
 whose restriction to each $A_i$ agrees with $g_i$, 
called the (\emph{piecewise}) \emph{sum of the family} $(g_i)_{i\in I}$. We also use $f+g$ to denote the 
sum of $f$ and $g$.  
Clearly, this operation is associative and commutative, and it is not difficult to see that
concatenation is distributive over sum,  
i.e. for any family $(g_i)_{i\in I}$ of maps on pairwise disjoint domains
and any map $f$
 \begin{displaymath}
({\Sigma }_{i\in I}g_i)f={\Sigma }_{i\in I}(g_if) \qquad \textrm{ and } \qquad 
f({\Sigma }_{i\in I}g_i)={\Sigma }_{i\in I}(fg_i). 
\end{displaymath}

Let $m$ and $n_j$, $j\in J$, be positive integers, and let $V$ be an arbitrary set disjoint from $\bf m$ and 
each ${\bf n}_j$. 
Any non-empty family $H=(h_j)_{j\in J}$ of maps $h_j:{\bf n}_j\rightarrow {\bf m}\cup V$
 is called a \emph{minor formation scheme} with \emph{target} $\bf m$,
 \emph{indeterminate set} $V$ and \emph{source family} $({\bf n}_j)_{j\in J}$.
Let $(R_j)_{j\in J}$ be a non-empty family of relations (of various arities) on the same set $A$,
 each $R_j$ of arity $n_j$.
An $m$-ary relation $R$ on $A$ is said to be a \emph{tight conjunctive minor} of 
the family $(R_j)_{j\in J}$ \emph{via the scheme $H$},
or simply a \emph{tight conjunctive minor} of the family $(R_j)_{j\in J}$, if 
for every $m$-tuple $\bf a$ in $A^m$, the following are equivalent:
\begin{itemize}
\item[(a)] ${\bf a}\in  R $;
\item[(b)] there is a map $\sigma :V\rightarrow A$ such that, for all $j$ in $J$, we have
$({\bf a}+\sigma )h_j\in R_j$.
\end{itemize}
The map $\sigma $ is called a \emph{Skolem map}. The $n_j$-tuple $({\bf a}+\sigma )h_j$ denotes 
the concatenation of the sum ${\bf a}+\sigma $ and $h_j$.
Formation of tight conjunctive minors subsumes permutation, identification,
projection and addition of dummy arguments, as well as arbitrary intersection of relations of the same arity. 

\smallskip

If for every $m$-tuple $\bf a$ in $A^m$, we have $(a)\Rightarrow (b)$, then $R$ is said to be a 
 \emph{restrictive conjunctive minor} of the family $(R_j)_{j\in J}$ \emph{via $H$},
or simply a \emph{restrictive conjunctive  minor} of the family $(R_j)_{j\in J}$.
On the other hand, if for every $m$-tuple $\bf a$ in $A^m$, we have $(b)\Rightarrow (a)$, then 
we say that $R$ is an \emph{extensive conjunctive minor} of the family $(R_j)_{j\in J}$ \emph{via $H$},
or simply an \emph{extensive conjunctive minor} of the family $(R_j)_{j\in J}$.
Thus a relation $R$ is a tight conjunctive minor of the family $(R_j)_{j\in J}$ if
 it is both a restrictive conjunctive minor and an extensive 
conjunctive minor of the family $(R_j)_{j\in J}$.

 \smallskip

An $A$-to-$B$ constraint $(R,S)$ is said to be a \emph{conjunctive minor} of a non-empty family 
$(R_j,S_j)_{j\in J}$ of $A$-to-$B$ constraints (of various arities) \emph{via a scheme $H$},
(or simply a \emph{conjunctive minor} of the family of constraints) if 
\begin{itemize}
\item[(i)] $R$ is a restrictive conjunctive minor of $(R_j)_{j\in J}$ via $H$, and 
\item[(ii)] $S$ is an extensive conjunctive minor of $(S_j)_{j\in J}$ via $H$. 
\end{itemize}
If the indeterminate set $V$ of the scheme $H$ is empty, i.e. for every $j$ in $J$, 
the maps $h_j$ are valued in $\bf m$, then $(R,S)$ is called a \emph{weak conjunctive minor} of the family 
$(R_j,S_j)_{j\in J}$ (see [C]).   
Observe that this operation subsumes in particular \emph{relaxations}: 
$(R,S)$ is said to be a \emph{relaxation} of $(R_0,S_0)$ if $R\subseteq R_0$ and $S\supseteq S_0$, and it is called a
 \emph{finite relaxation}, if $R$ is finite.
If both $R$ and $S$ are tight conjunctive minors of the respective families $(R_j)_{j\in J}$ and 
$(S_j)_{j\in J}$ (on $A$ and $B$,respectively) via the same scheme $H$, the constraint   
$(R,S)$ is said to be a \emph{tight conjunctive minor} of the family $(R_j,S_j)_{j\in J}$ \emph{via $H$}, 
or simply a \emph{tight conjunctive minor} of the family of constraints. In this case, if in addition
$\mid J\mid =1$, say $J=\{0\}$, then the family $(R_j,S_j)_{j\in J}$ 
contains a single constraint $(R_0,S_0)$, and $(R,S)$ is said to be a \emph{simple minor}
 of $(R_0,S_0)$ (see [Pi2]).
The following is a special case of Claim 1 in the proof of Theorem 2 in [CF1]:

\begin{main1}
If $(R,S)$ is a conjunctive minor of a non-empty family $(R_j,S_j)_{j\in J}$ 
of $A$-to-$B$ constraints, and, for each $j\in J$, $(R_j,S_j)$ is a conjunctive minor of a non-empty family 
$(R_{j}^i,S_{j}^i)_{i\in I_j}$,
then $(R,S)$ is a conjunctive minor of the non-empty family $(R_{j}^i,S_{j}^i)_{j\in J,i\in I_j}$.
\end{main1}

We say that a set $\mathcal{T}$ of relational constraints is 
\emph{closed under formation of conjunctive minors}
if whenever every member of a non-empty family $(R_j,S_j)_{j\in J}$ of constraints
 is in $\mathcal{T}$, all conjunctive minors of the family $(R_j,S_j)_{j\in J}$ are also in $\mathcal{T}$. 
For any set of constraints $\mathcal{T}$, we denote by ${\bf CM}(\mathcal{T})$ the smallest 
set of constraints containing $\mathcal{T}$, the binary equality constraint and the empty constraint,
and closed under formation of conjunctive minors. 
By the Transitivity Lemma it follows that ${\bf CM}(\mathcal{T})$ is the set of all  
conjunctive minors of non-empty families of $A$-to-$B$ constraints in  
$\mathcal{T}\cup \{(=_A,=_B),(\emptyset ,\emptyset )\}$, and that 
${\bf CM}({\bf CM}(\mathcal{T}))={\bf CM}(\mathcal{T})$.   

\begin{fact} The operator $\mathcal{T}\mapsto {\bf CM}(\mathcal{T})$ is a closure operator on 
the set of all $A$-to-$B$ relational constraints.
\end{fact}

In analogy with classes of external operations, 
we need to consider a further condition for the characterization of the closed sets of constraints.
A set $\mathcal{T}$ of relational constraints is said to be \emph{locally closed} if $\mathcal{T}$ 
contains every $A$-to-$B$ constraint $(R,S)$ such that the set of all its finite relaxations 
 is contained in $\mathcal{T}$.
The \emph{local closure} of a set $\mathcal{T}$ of relational constraints, 
denoted by ${\bf LO}(\mathcal{T})$, is the smallest locally closed
set of constraints containing $\mathcal{T}$. In other words, ${\bf LO}(\mathcal{T})$ is the set of
constraints obtained from $\mathcal{T}$ by adding all those constraints whose finite relaxations
are all in $\mathcal{T}$, and thus we have:

\begin{fact} The operator $\mathcal{T}\mapsto {\bf LO}(\mathcal{T})$ is a closure operator on 
the set of all $A$-to-$B$ relational constraints.
\end{fact}
As in the case of function classes, if $A$ is finite, then every set
of $A$-to-$B$ constraints is locally closed. 
The following result provides the characterization of the closed sets of constraints with respect
to the Galois connection ${\bf FSC}-{\bf CSF}$:

\begin{theorem}\emph{([CF1]:)} 
Consider arbitrary non-empty sets $A$ and $B$. A set $\mathcal{T}$ of $A$-to-$B$ relational constraints 
is characterized by some set of $B$-valued functions on $A$ if and only if it is
locally closed and contains the binary equality constraint,
 the empty constraint, and it is closed under formation of conjunctive minors.
\end{theorem}

We finish this Section with a description of the closure operators ${\bf FSC}\circ {\bf CSF}$ and 
${\bf CSF}\circ {\bf FSC}$ as compositions of the operators ${\bf Lo}$ and ${\bf VS}$, and 
${\bf LO}$ and ${\bf CM}$, respectively. The statements $(a)$ and $(b)$ below are analogues
of $(iii)$ in Lemma 2.5 and Proposition 3.8, respectively, in [P\" o2]:

\begin{theorem}
Consider arbitrary non-empty sets $A$ and $B$, and let $\mathcal{K}\subseteq \cup _{n\geq 1}B^{A^n}$ 
be a class of $B$-valued functions on $A$, and $\mathcal{T}$ 
 a set of $A$-to-$B$ relational constraints. The following hold:
\begin{itemize}
\item[(a)] If ${\bf VS}(\mathcal{K})=\mathcal{K}$,
 then ${\bf VS}({\bf Lo}(\mathcal{K}))={\bf Lo}(\mathcal{K})$; 
\item[(b)] If ${\bf CM}(\mathcal{T})=\mathcal{T}$, then
 ${\bf CM}({\bf LO}(\mathcal{T}))={\bf LO}(\mathcal{T})$.
\end{itemize}
\end{theorem}
 
\begin{proof}
 First we prove $(a)$.
 Suppose that $g$ is an $t$-ary function in ${\bf VS}({\bf Lo}(\mathcal{K}))$.
That is, there is an $n$-ary function $f$ in ${\bf Lo}(\mathcal{K})$, 
and $t$-ary projections $p_1,\ldots ,p_n\in \mathcal{O}_A$ such that 
$g=f(p_1,\ldots ,p_n)$. To prove that $g$ belongs to ${\bf Lo}(\mathcal{K})$, we show
that, for every finite subset $F$ of $A^t$, there is a an $t$-ary function $g_F$
 in $\mathcal{K}$ such that
$g({\bf a})=g_F({\bf a})$ for every ${\bf a}\in F$.
So let $F$ be any finite subset of $A^t$, and consider the finite subset $F'\subseteq A^n$ defined by
\begin{displaymath}
F'=\{(p_1({\bf a}),\ldots ,p_n({\bf a}))\mid {\bf a}\in F\} 
\end{displaymath}
From the fact $f\in {\bf Lo}(\mathcal{K})$, it follows that there is an $n$-ary function $f_{F'}$ in 
$\mathcal{K}$ such that $f({\bf a}')=f_{F'}({\bf a}')$, for every ${\bf a}'\in F'$.
Consider the $t$-ary function $g_F$ defined by 
$g_F=f_{F'}(p_1,\ldots ,p_n)$. Note that $g_F$ belongs to $\mathcal{K}$, because
${\bf VS}(\mathcal{K})=\mathcal{K}$. 
By the definition of $f_{F'}$ and $g_F$, we have that, for every $t$-tuple ${\bf a}\in F$, 
\begin{displaymath}
g({\bf a})=f(p_1,\ldots ,p_n)({\bf a})=f_{F'}(p_1,\ldots ,p_n)({\bf a})=g_F({\bf a})
\end{displaymath}
Since the above argument works for every finite subset $F$ of $A^t$, we have that $g$
 is in ${\bf Lo}(\mathcal{K})$.

\bigskip 

To prove $(b)$, we show that every constraint in ${\bf CM}({\bf LO}(\mathcal{T}))$ is also in 
${\bf LO}(\mathcal{T})$. 
Note that the binary equality constraint and the empty constraint are in ${\bf LO}(\mathcal{T})$.
Thus ${\bf CM}({\bf LO}(\mathcal{T}))$ is the set of all  
conjunctive minors of non-empty families of $A$-to-$B$ constraints in  
${\bf LO}(\mathcal{T})$.
So let $(R,S)$ be a conjunctive minor of a non-empty family $(R_j,S_j)_{j\in J}$
 of constraints in ${\bf LO}(\mathcal{T})$ via a scheme $H$ with indeterminate set $V$. Consider the tight 
conjunctive minor $(R_0,S_0)$ of the family $(R_j,S_j)_{j\in J}$ via the same scheme $H=(h_j)_{j\in J}$.
Note that every relaxation of $(R,S)$ is a relaxation of $(R_0,S_0)$. Thus to prove that
$(R,S)\in {\bf LO}(\mathcal{T})$, it is enough to show that every finite relaxation of
 $(R_0,S_0)$ is in $\mathcal{T}$, because it follows then
that every finite relaxation of $(R,S)$ is in $\mathcal{T}$. 

\smallskip

Let $(F,S')$ be a finite relaxation of $(R_0,S_0)$,
 say $F$ having $n$ distinct elements ${\bf a}_1, \ldots ,{\bf a}_n$.
 Since $F\subseteq R_0$ and $R_0$ is a tight conjunctive minor 
of the family $(R_j)_{j\in J}$ via $H$, we have that, for every ${\bf a}_i\in F$,
there is a Skolem map ${\sigma }_i:V\rightarrow A$ such that, for all $j$ in $J$, 
$({\bf a}_i+{\sigma }_i)h_j\in R_j$.
For each $j$ in $J$, let $F_j$ be the subset of $R_j$, given by  
\begin{displaymath}
F_j=\{({\bf a}_i+{\sigma }_i)h_j\mid {\bf a}_i\in F\}. 
\end{displaymath}
Consider the non-empty family $(F_j,S_j)_{j\in J}$ of constraints with finite antecedents $F_j$.
Clearly, $(F,S')$ is a conjunctive minor of the family $(F_j,S_j)_{j\in J}$, and for each
$j$ in $J$, $(F_j,{S}_j)$ is a relaxation of $(R_j,S_j)$. 
Since ${\bf CM}(\mathcal{T})=\mathcal{T}$, and for each $j$ in $J$,
$(R_j,S_j)$ is in ${\bf LO}(\mathcal{T})$,   
we have that every member of the family $(F_j,S_j)_{j\in J}$ belongs to $\mathcal{T}$.  
 Hence $(F,S')$ 
is a conjunctive minor of a family of members of $\mathcal{T}$, and thus $(F,S')$  
is also in $\mathcal{T}$. 
\end{proof}

From Theorem 1, Theorem 2 and Theorem 3, we get the following factorization 
of the closure operators ${\bf FSC}\circ {\bf CSF}$ and 
${\bf CSF}\circ {\bf FSC}$: 

\begin{theorem}
Consider arbitrary non-empty sets $A$ and $B$.
 For any class of functions $\mathcal{K}\subseteq \cup _{n\geq 1}B^{A^n}$ and 
any set $\mathcal{T}$ of $A$-to-$B$ relational constraints, we have:
\begin{itemize}
\item[(i)] ${\bf FSC}({\bf CSF}(\mathcal{K}))={\bf Lo}({\bf VS}(\mathcal{K}))$, and
\item[(ii)] ${\bf CSF}({\bf FSC}(\mathcal{T}))={\bf LO}({\bf CM}(\mathcal{T}))$.
\end{itemize}
\end{theorem}

\section{Galois connections between functions and constraints with arity restrictions}

Let $n$ and $m$ be positive integers.
For any set $\mathcal{T}$ of $A$-to-$B$ constraints, we denote by ${\bf FSC}_n(\mathcal{T})$ 
 the class of all $n$-ary functions satisfying every member of $\mathcal{T}$, and
 for any class $\mathcal{K}$ 
of $B$-valued functions on $A$, we denote by ${\bf CSF}_m(\mathcal{K})$ the
set of all $m$-ary constraints satisfied by every member of $\mathcal{K}$. 
That is, 
\begin{itemize}
\item[$\cdot $] ${\bf FSC}_n(\mathcal{T})=B^{A^n}\cap {\bf FSC}(\mathcal{T})$, and  
\item[$\cdot $] ${\bf CSF}_m(\mathcal{K})= \mathcal{Q}_m\cap {\bf CSF}(\mathcal{K})$,
where $\mathcal{Q}_m$ denotes the set of all $m$-ary $A$-to-$B$ constraints, i.e. the cartesian product 
 $\mathcal{P}(A^m)\times \mathcal{P}(B^m)$ of the set of all subsets of $A^m$ and the set of all subsets of $B^m$.
\end{itemize}
Thus a class $\mathcal{K}_n\subseteq B^{A^n}$ of $n$-ary $B$-valued functions on $A$ 
 is said to be \emph{definable within $B^{A^n}$} by a set $\mathcal{T}$ of $A$-to-$B$ constraints,
 if $\mathcal{K}_n={\bf FSC}_n(\mathcal{T})$, and a set $\mathcal{T}_m$ of $m$-ary $A$-to-$B$ constraints
is said to be \emph{characterized within $\mathcal{Q}_m$} by a set $\mathcal{K}$ of 
$B$-valued functions on $A$, if $\mathcal{T}_m={\bf CSF}_m(\mathcal{K})$.

\subsection{Restricting function arities}

We begin with the characterization of the closed classes of functions of fixed arities definable by 
relational constraints, and the description
of the dual closed sets characterized by functions of given arities.
A class $\mathcal{K}_n$ of $n$-ary $B$-valued functions on $A$  
is said to be \emph{closed under $n$-ary simple variable substitutions}
if every $n$-ary function obtained from a member of $\mathcal{K}_n$
 by simple variable substitution also belongs to $\mathcal{K}_n$, that is,
if $\mathcal{K}_n=B^{A^n}\cap {\bf VS}(\mathcal{K}_n)$. We denote by ${\bf VS}_n(\mathcal{K}_n)$ 
the \emph{closure under $n$-ary simple variable substitutions of} $\mathcal{K}_n$ given by
${\bf VS}_n(\mathcal{K}_n)=B^{A^n}\cap {\bf VS}(\mathcal{K}_n)$.  
Note that if $\mathcal{K}$ is a locally closed class of $B$-valued functions on $A$,
and closed under simple variable substitutions, and if $\mathcal{K}_n$
is the class of $n$-ary functions in $\mathcal{K}$, then
$\mathcal{K}_n$ is locally closed and it is closed under $n$-ary simple variable substitutions.
The following is an immediate consequence of the definitions above:

\begin{fact} Consider arbitrary non-empty sets $A$ and $B$, and let $n$ be a positive integer.
 For any class $\mathcal{K}_n$ of $n$-ary $B$-valued functions on $A$, 
  \begin{displaymath}
 B^{A^n}\cap {\bf Lo}({\bf VS}(\mathcal{K}_n))={\bf Lo}({\bf VS}_n(\mathcal{K}_n)).
\end{displaymath}   
\end{fact}

We make use of Fact 5 to prove:

\begin{theorem}
Consider arbitrary non-empty sets $A$ and $B$, and let $n$ be a positive integer.
 For any class of $n$-ary functions $\mathcal{K}_n\subseteq B^{A^n}$ 
the following conditions are equivalent:
\begin{itemize}
\item[(i)] $\mathcal{K}_n$ is locally closed and it is closed under $n$-ary simple variable substitutions;
\item[(ii)] $\mathcal{K}_n$ is definable within $B^{A^n}$ by some set of $A$-to-$B$ constraints.
\end{itemize}
\end{theorem}

\begin{proof}
To prove $(ii)\Rightarrow (i)$, assume $(ii)$, i.e. $\mathcal{K}_n={\bf FSC}_n(\mathcal{T})$, for
some set $\mathcal{T}$ of $A$-to-$B$ constraints.
 Let $\mathcal{K}={\bf FSC}(\mathcal{T})$. 
 By Theorem 1, we have that
$\mathcal{K}$ is locally closed and it is closed under simple variable substitutions, and since 
$\mathcal{K}_n=B^{A^n}\cap \mathcal{K}$, it follows from the comment preceeding Fact 5 that
$\mathcal{K}_n$ is locally closed and it is closed under $n$-ary 
simple variable substitutions.

\bigskip

To show that $(i)\Rightarrow (ii)$ holds, assume $(i)$, and let $\mathcal{K}={\bf VS}(\mathcal{K}_n)$.
Since ${\bf Lo}(\mathcal{K})$ is closed under simple variable substitutions,
 it follows from Theorem 1,
 that ${\bf Lo}(\mathcal{K})$ is definable by some set $\mathcal{T}$ of $A$-to-$B$ constraints.
 By Fact 5 $\mathcal{K}_n$ is the class of $n$-ary functions 
in ${\bf Lo}(\mathcal{K})$, and thus $\mathcal{K}_n$ is definable within $B^{A^n}$ by $\mathcal{T}$.     
\end{proof}

Note that for $n=1$, every class $\mathcal{K}\subseteq B^A$ of unary $B$-valued functions on $A$ is 
closed under unary simple variable substitutions. Thus,
from Theorem 5, it follows:

\begin{cor}
Consider arbitrary non-empty sets $A$ and $B$.
A class $\mathcal{K}$ of unary $B$-valued functions on $A$ 
is definable within $B^{A}$ by some set of $A$-to-$B$ constraints if and only if 
 $\mathcal{K}$ is locally closed.
\end{cor}

Theorem 5 provides necessary and sufficient closure conditions for a class of external operations of fixed arity
to be definable by relational constraints. To describe the closed sets of relational constraints characterized 
by external operations of a given arity,
 we need to strengthen the notion of local closure for sets of constraints.

\smallskip

For a positive integer $n$, we say that a set $\mathcal{T}$ of relational constraints is 
$n$-\emph{locally closed} if $\mathcal{T}$ contains
every $A$-to-$B$ constraint $(R,S)$ such that the set of all its relaxations 
with antecedent of size at most $n$ is contained in $\mathcal{T}$.
The $n$-\emph{local closure} of a set $\mathcal{T}$ of relational constraints
 is the smallest $n$-locally closed
set of constraints containing $\mathcal{T}$, and it is 
denoted by ${\bf LO}_n(\mathcal{T})$. Note that 
every $n$-locally closed set of constraints is in particular locally closed. In fact,
for any set $\mathcal{T}$ of $A$-to-$B$ relational constraints,
 ${\bf LO}(\mathcal{T})=\cap _{m\geq 1}{\bf LO}_m(\mathcal{T})$.
Similarly to the closure ${\bf LO}(\mathcal{T})$,
it is easy to see that ${\bf LO}_n(\mathcal{T})$ 
is the set of constraints obtained from $\mathcal{T}$ by adding all those constraints
 whose finite relaxations with 
antecedent of size at most $n$ are all in $\mathcal{T}$. From these observations it follows:
\begin{fact}
Consider arbitrary non-empty sets $A$ and $B$, and let $n$ be a positive integer.
\begin{itemize}
\item[(a)] The operator $\mathcal{T}\mapsto {\bf LO}_n(\mathcal{T})$ is a closure operator on 
the set of all $A$-to-$B$ relational constraints.
\item[(b)] For any set $\mathcal{T}$ of $A$-to-$B$ relational constraints, 
$({\bf LO}_n(\mathcal{T}))_{n\geq 1}$ is a descending chain under inclusion, i.e.
 ${\bf LO}_m(\mathcal{T})\subseteq {\bf LO}_n(\mathcal{T})$ whenever $m\geq n$, 
and its infimum is ${\bf LO}(\mathcal{T})$.
\end{itemize}
\end{fact}
 
The following analogue of Theorem 2 shows that, in addition, parametrized local closure guarantees the existence
of characterizations of sets of constraints by classes of functions of fixed arities. 
\begin{theorem}
Consider arbitrary non-empty sets $A$ and $B$ and let $n$ be a positive integer.
Let $\mathcal{T}$ be a set of $A$-to-$B$ relational constraints. Then the following are equivalent:
\begin{itemize}
\item[(i)]$\mathcal{T}$ is $n$-locally closed and contains the binary equality constraint,
 the empty constraint, and it is closed under formation of conjunctive minors;
\item[(ii)]$\mathcal{T}$ is characterized by some set of $n$-ary $B$-valued functions on $A$.
\end{itemize}
\end{theorem}

\begin{proof} To show that $(ii)\Rightarrow (i)$, assume $(ii)$.
From Theorem 2, it follows that $\mathcal{T}$ contains the binary equality constraint,
 the empty constraint, and it is closed under formation of conjunctive minors.
Thus to show that $(ii)\Rightarrow (i)$ holds, we only have to prove that $\mathcal{T}$ 
is $n$-locally closed.
Let $(R,S)$ be an $m$-ary constraint not in $\mathcal{T}$.  
 From $(ii)$, it follows that there is an $n$-ary function $f$ satisfying every constraint in 
$\mathcal{T}$ but not $(R,S)$, i.e. there are ${\bf a}^1,\ldots ,{\bf a}^n \in R$ such that 
$f({\bf a}^1\ldots {\bf a}^n)\not\in S$. Let $F=\{{\bf a}^1,\ldots ,{\bf a}^n \}$. 
 Clearly, the constraint $(F,S)$ is a relaxation of $(R,S)$ with antecedent of size at most $n$, 
which is not satisfied by $f$. Hence $(F,S)$ does not belong to $\mathcal{T}$ .

\bigskip

To prove the implication $(i)\Rightarrow (ii)$, we show that for each constraint $(R,S)$
 not in $\mathcal{T}$, 
there is an $n$-ary function satisfying every constraint in $\mathcal{T}$, but not $(R,S)$.

Suppose that $(R,S)$ does not belong to $\mathcal{T}$.
 Since $\mathcal{T}$ is $n$-locally closed, we know that there is a relaxation $(F,S')$ of $(R,S)$, with 
finite antecedent of size $m\leq n$, which does not belong to $\mathcal{T}$.
Also, by Fact 6 $(b)$ it follows that $\mathcal{T}$ is locally closed.
 Since $\mathcal{T}$ also contains the binary equality constraint,
 the empty constraint, and it is closed under formation of conjunctive minors, it follows from    
Theorem 2 that $\mathcal{T}$ is characterized by some set of $B$-valued functions on $A$.
Let $g$ be a function separating $(F,S')$ from $\mathcal{T}$, i.e. $g$ satisfies every constraint in 
$\mathcal{T}$, but not $(F,S')$. Note that $F$ has size $m\leq n$. 
Thus, by identification of variables and addition of inessential variables, we can obtain from $g$ a
separating function $g'$ of arity $n$, and 
the proof of implication $(ii)\Rightarrow (i)$ is complete. 
\end{proof}

We say that a set $\mathcal{T}$ of relational constraints is 
\emph{closed under arbitrary unions} if $(\cup  _{i\in I}R_i,\cup _{i\in I}S_i)$ is in $\mathcal{T}$,
whenever $(R_i,S_i)_{i\in I}$ is a non-empty family of members of $\mathcal{T}$.
Closure under arbitrary unions is closely related to the notion of $1$-local closure:

\begin{proposition}
If $\mathcal{T}$ is a set of relational constraints closed under taking relaxations,
then $\mathcal{T}$ is closed under arbitrary unions if and only if it is $1$-locally closed.  
\end{proposition}

\begin{proof}
Clearly, every set of relational constraints closed under arbitrary unions is $1$-locally closed.
For the converse, let $(R_i,S_i)_{i\in I}$ be a non-empty family of members of $\mathcal{T}$.
Since $\mathcal{T}$ is closed under taking relaxations, we have that 
$(\{r\},\cup _{i\in I}S_i)$ belongs to $\mathcal{T}$ for every $r$ in $\cup _{i\in I}R_i$. 
By $1$-local closure, we conclude that
 $(\cup  _{i\in I}R_i,\cup _{i\in I}S_i)$ is in $\mathcal{T}$.  
\end{proof}

Using Proposition 1, we obtain as a particular case of Theorem 6 the following description 
of the sets of constraints characterized by unary functions. 
 
\begin{cor}
Consider arbitrary non-empty sets $A$ and $B$.
Let $\mathcal{T}$ be a set of $A$-to-$B$ relational constraints. Then the following are equivalent:
\begin{itemize}
\item[(i)]$\mathcal{T}$ contains the binary equality constraint and the empty constraint,
 and it is closed under arbitrary unions and closed under formation of conjunctive minors;
\item[(ii)]$\mathcal{T}$ is characterized by some set of unary $B$-valued functions on $A$.
\end{itemize}
\end{cor}

The closure operators associated with the Galois connection ${\bf FSC}_n-{\bf CSF}$,
have decompositions analogous to those given in Theorem 4.
To establish them, one needs the following (statement $(b)$ in Theorem 7 below is the analogue of
 Proposition 3.8 $(ii)$ in [P\" o2] concerning sets of relations):

\begin{theorem}
Consider arbitrary non-empty sets $A$ and $B$, and for a positive integer
$n$, let $\mathcal{K}_n\subseteq B^{A^n}$
be a class of $n$-ary functions, and $\mathcal{T}$ 
 a set of $A$-to-$B$ relational constraints. The following hold:
\begin{itemize}
\item[(a)] If $\mathcal{K}_n={\bf VS}_n(\mathcal{K}_n)$, 
 then ${\bf VS}_n({\bf Lo}(\mathcal{K}_n))={\bf Lo}(\mathcal{K}_n)$;
\item[(b)] If ${\bf CM}(\mathcal{T})=\mathcal{T}$, then
${\bf CM}({\bf LO}_n(\mathcal{T}))={\bf LO}_n(\mathcal{T})$.
 \end{itemize}
\end{theorem}

\begin{proof}
First we prove $(a)$. By $(a)$ of Theorem 3 it follows that
\begin{displaymath}
{\bf VS}({\bf Lo}({\bf VS}(\mathcal{K}_n)))={\bf Lo}({\bf VS}(\mathcal{K}_n)) 
\end{displaymath} 
and therefore 
\begin{displaymath}
B^{A^n}\cap {\bf VS}({\bf Lo}({\bf VS}(\mathcal{K}_n)))=B^{A^n}\cap {\bf Lo}({\bf VS}(\mathcal{K}_n)).  
\end{displaymath}
Clearly, ${\bf VS}_n({\bf Lo}(\mathcal{K}_n))\subseteq B^{A^n}\cap {\bf VS}({\bf Lo}({\bf VS}(\mathcal{K}_n)))$.
By Fact 5,
\begin{displaymath}
B^{A^n}\cap {\bf Lo}({\bf VS}(\mathcal{K}_n))={\bf Lo}({\bf VS}_n(\mathcal{K}_n))
\end{displaymath}
and since $\mathcal{K}_n={\bf VS}_n(\mathcal{K}_n)$, we have 
${\bf Lo}({\bf VS}_n(\mathcal{K}_n))={\bf Lo}(\mathcal{K}_n)$. Hence,
\begin{displaymath}
{\bf VS}_n({\bf Lo}(\mathcal{K}_n))\subseteq {\bf Lo}(\mathcal{K}_n)\subseteq {\bf VS}_n({\bf Lo}(\mathcal{K}_n))
\end{displaymath}
 i.e. ${\bf VS}_n({\bf Lo}(\mathcal{K}_n))={\bf Lo}(\mathcal{K}_n)$. 
   
A proof of $(b)$ in Theorem 7 is obtained essentially by replacing, in the proof of $(b)$ of Theorem 3,
${\bf LO}$ by ${\bf LO}_n$, and ``finite relaxation" by ``finite relaxation with antecedent of size at most $n$". The key observation is that
$\mid F_j\mid \leq \mid F\mid \leq n$.
\end{proof}

 From Theorem 5, Theorem 6 and Theorem 7, we obtain factorizations of the closure operators
 ${\bf FSC}_n\circ {\bf CSF}$ and ${\bf CSF}\circ {\bf FSC}_n$, as compositions of the operators
 ${\bf Lo}$ and ${\bf VS}_n$, and 
${\bf LO}_n$ and ${\bf CM}$, respectively:

\begin{theorem}
Consider arbitrary non-empty sets $A$ and $B$, and let $n$ be a positive integer.
 For any class of  $n$-ary functions $\mathcal{K}_n\subseteq B^{A^n}$ and 
any set $\mathcal{T}$ of $A$-to-$B$ relational constraints, the following hold:
\begin{itemize}
\item[(i)] ${\bf FSC}_n({\bf CSF}(\mathcal{K}_n))={\bf Lo}({\bf VS}_n(\mathcal{K}_n))$, and
\item[(ii)] ${\bf CSF}({\bf FSC}_n(\mathcal{T}))={\bf LO}_n({\bf CM}(\mathcal{T}))$.
\end{itemize}
\end{theorem}

\subsection{Restricting constraint arities}

We now consider arity restrictions on sets of relational constraints.
 First we determine necessary and sufficient
closure conditions for function class definability by sets of constraints of fixed arity.
The following parameterized notion of local closure corresponds to that appearing in [P$\ddot o$2], for
operations on a given set. 
For a positive integer $m$, a class $\mathcal{K}$ of $B$-valued functions on $A$
 is said to be $m$-\emph{locally closed} if for every $B$-valued function $f$ on $A$
 the following holds: if every restriction of $f$ to a finite subset $D\subseteq A^n$ of size at most $m$,
 coincides with the restriction to $D$ of some member of $\mathcal{K}$, then $f$ belongs to $\mathcal{K}$.
(See [FH1] and [FH2] for two different but somewhat related notions of $m$-local closure defined 
on classes of \emph{pseudo-Boolean functions}, i.e. maps of the form $\{0,1\}^n\rightarrow {\bf R}$, where
${\bf R}$ denotes the field of real numbers.)
 For any class of functions $\mathcal{K}\subseteq \cup _{n\geq 1}B^{A^n}$
  the smallest $m$-locally closed class of functions containing $\mathcal{K}$, 
which we denote by ${\bf Lo}_m(\mathcal{K})$, is called the \emph{$m$-local closure} of $\mathcal{K}$,
and it is the class obtained from 
$\mathcal{K}$ by adding all those functions whose restriction to each subset of
its domain $A^n$ of size at most $m$ coincides with a restriction of some member of $\mathcal{K}$.
The following summarizes some immediate consequences of the definitions and the above observations.
\begin{fact}
Consider arbitrary non-empty sets $A$ and $B$, and let $m$ be a positive integer.
\begin{itemize}
\item[(a)] The operator $\mathcal{K}\mapsto {\bf Lo}_m(\mathcal{K})$ is a closure operator on $\cup _{n\geq 1}B^{A^n}$.
\item[(b)] For any class $\mathcal{K}\subseteq \cup _{n\geq 1}B^{A^n}$, we have
 ${\bf Lo}_n(\mathcal{T})\subseteq {\bf Lo}_m(\mathcal{T})$ whenever $n\geq m$, and 
${\bf Lo}(\mathcal{K})=\cap _{n\geq 1}{\bf Lo}_n(\mathcal{K})$.
Thus every $m$-locally closed class of functions is in particular locally closed.
\end{itemize}
\end{fact}

As in the case of sets of relational constraints, it turns out that this parametrized notion of local closure,
together with the conditions given by Theorem 2, suffices to characterize the classes of functions definable 
by sets of constraints of fixed arities.

\begin{theorem}
Consider arbitrary non-empty sets $A$ and $B$ and let $m$ be a positive integer.
 For a class of functions $\mathcal{K}\subseteq \cup _{n\geq 1}B^{A^n}$ 
the following conditions are equivalent:
\begin{itemize}
\item[(i)] $\mathcal{K}$ is $m$-locally closed and it is closed under simple variable substitutions;
\item[(ii)] $\mathcal{K}$ is definable by some set of $A$-to-$B$ $m$-ary constraints.
\end{itemize}
\end{theorem}

\begin{proof}
To prove the implication $(ii)\Rightarrow (i)$, assume $(ii)$. From Theorem 1,
it follows that $\mathcal{K}$ is closed under simple variable substitutions. 
 To see that $\mathcal{K}$ is $m$-locally closed, let $f$ be an $n$-ary function not in $\mathcal{K}$, 
and let $(R,S)$ be an $A$-to-$B$ $m$-ary constraint satisfied by every function $g$ in $\mathcal{K}$ but
 not satisfied by $f$.
 Hence, for some ${\bf a}^1,\ldots ,{\bf a}^n \in R$, we have 
$f({\bf a}^1\ldots {\bf a}^n)\not\in S$, and $g({\bf a}^1\ldots {\bf a}^n)\in S$,
 for every $n$-ary function $g$ in $\mathcal{K}$.
 Let $F=\{({\bf a}^1(i),\ldots ,{\bf a}^n (i)):i\in {\bf m}\}$.
Clearly, the restriction of $f$ to the set $F$, which has size at most $m$, does not coincide
 with that of any member of $\mathcal{K}$.

\bigskip

Now we prove the implication $(i)\Rightarrow (ii)$. If $\mathcal{K}=\emptyset $,
 then the single constraint $(A^m,\emptyset )$ clearly
defines $\mathcal{K}$. Hence, we may assume that $\mathcal{K}$ is non-empty.
 Consider a function $g\not\in \mathcal{K}$, say of arity $n$.
Thus there is a restriction $g_F$ of $g$ to a non-empty finite subset $F\subseteq A^n$ of size $p\leq m$
 which does not agree with any function in $\mathcal{K}$ restricted to $F$. 

Let ${\bf a}^1,\ldots ,{\bf a}^n$ be any $m$-tuples in $A^m$,
such that $F=\{({\bf a}^1(i),\ldots ,{\bf a}^n (i)):i\in {\bf m}\}$.
Let $(R,S)$ be the $m$-ary constraint whose antecedent is 
$R=\{{\bf a}^1,\ldots ,{\bf a}^n\}$, and whose consequent is given by
$S=\{f({\bf a}^1\ldots {\bf a}^n):f\in \mathcal{K}_n\}$, where $\mathcal{K}_n$
denotes the set of $n$-ary functions in $\mathcal{K}$.   
It follows from the definition of $R$ and $S$ that $(R,S)$ is an $A$-to-$B$ 
$m$-ary constraint, that $g$ does not satisfy $(R,S)$,
 and, since $\mathcal{K}$ is closed under simple variable substitutions,
 that every function in $\mathcal{K}$ satisfies $(R,S)$. 
\end{proof}

Now we describe the closed sets of constraints of fixed arities characterized by the 
functions of several variables satisfying them. 
Let $\mathcal{T}_m$ be a set of $A$-to-$B$ $m$-ary relational constraints. 
We say that $\mathcal{T}_m$ is \emph{closed under formation of $m$-ary conjunctive minors}
if whenever every member of a non-empty family $(R_j,S_j)_{j\in J}$ of constraints
 is in $\mathcal{T}_m$, all $m$-ary conjunctive minors of the
 family are also in $\mathcal{T}_m$.  

\bigskip

For a positive integer $m$, we refer to the constraint whose antecedent and consequent consists of all
$m$-tuples with all arguments equal, as the \emph{$m$-ary equality constraint}. Note that, for $2\leq m$,
the $m$-ary equality constraint is a tight conjunctive minor of a family of constraints with 
$m-1$ binary equality constraints, and, for $m>1$, the binary equality constraint is 
a tight conjunctive minor of the $m$-ary equality constraint.  
For any set $\mathcal{T}_m$ of $m$-ary constraints, let ${\bf CM}_m(\mathcal{T}_m)$ denote the smallest 
set of constraints containing $\mathcal{T}_m$, closed under formation of $m$-ary conjunctive minors,
 and containing the $m$-ary equality constraint and the empty constraint.
By the Transitivity Lemma it follows that ${\bf CM}_m(\mathcal{T}_m)=\mathcal{Q}_m\cap {\bf CM}(\mathcal{T}_m)$,
where $\mathcal{Q}_m$ denotes the set of all $A$-to-$B$ $m$-ary relational constraints.

\begin{lemma} Consider arbitrary non-empty sets $A$ and $B$. 
For any set $\mathcal{T}_m$ of $m$-ary constraints, 
  \begin{displaymath}
 \mathcal{Q}_m\cap {\bf LO}({\bf CM}(\mathcal{T}_m))={\bf LO}({\bf CM}_m(\mathcal{T}_m)),
\end{displaymath} 
where $\mathcal{Q}_m$ denotes the set of all $A$-to-$B$ $m$-ary relational constraints. 
\end{lemma} 

\begin{proof}
It is easy to verify that for any set $\mathcal{T}$ of relational constraints,
  \begin{displaymath}
 \mathcal{Q}_m\cap {\bf LO}(\mathcal{T})=
{\bf LO}(\mathcal{Q}_m\cap \mathcal{T}).
\end{displaymath} 
 By the remark preceding the lemma, it follows that for any set $\mathcal{T}_m$ of $m$-ary constraints, 
  \begin{displaymath}
 \mathcal{Q}_m\cap {\bf LO}({\bf CM}(\mathcal{T}_m))=
{\bf LO}(\mathcal{Q}_m\cap {\bf CM}(\mathcal{T}_m))={\bf LO}({\bf CM}_m(\mathcal{T}_m)).
\end{displaymath} 
\end{proof}

From the above definitions, one can easily verify that the following also holds:

\begin{fact}
Consider arbitrary non-empty sets $A$ and $B$. 
If $\mathcal{T}$ is a locally closed set of $A$-to-$B$ relational constraints,  
closed under formation of conjunctive minors, and 
$\mathcal{T}_m$ is the set of all $m$-ary contraints in $\mathcal{T}$, then 
$\mathcal{T}_m$ is locally closed, and 
closed under formation of $m$-ary conjunctive minors. 
\end{fact}

We use Lemma 1 and Fact 8 to prove the following, which provides necessary and sufficient 
conditions for a set of constraints of a given arity to be characterized by external operations:

\begin{theorem}
Consider arbitrary non-empty sets $A$ and $B$ and let $m$ be a positive integer.
Let $\mathcal{Q}_m$ be the set of all $A$-to-$B$ $m$-ary relational constraints, and let
$\mathcal{T}_m\subseteq \mathcal{Q}_m$. 
Then the following are equivalent:
\begin{itemize}
\item[(i)]$\mathcal{T}_m$ is locally closed, contains the $m$-ary equality constraint and 
the $m$-ary empty constraint, and it is
closed under formation of $m$-ary conjunctive minors;
\item[(ii)]$\mathcal{T}_m$ is characterized within $\mathcal{Q}_m$ by some set 
of $B$-valued functions on $A$.
\end{itemize}
\end{theorem}

\begin{proof}
To see that implication $(ii)\Rightarrow (i)$ holds, 
let $\mathcal{K}\subseteq \cup _{n\geq 1}B^{A^n}$ 
be a set of functions such that $\mathcal{T}_m={\bf CSF}_m(\mathcal{K})$. 
By Theorem 2, we have that
${\bf CSF}(\mathcal{K})$ is locally closed, contains the binary equality constraint and 
the empty constraint, and it is closed under formation of conjunctive minors.
 Hence, by Fact 8 $\mathcal{T}_m$ is locally closed, contains the $m$-ary equality constraint and 
$m$-ary empty constraint, and it is
closed under formation of $m$-ary conjunctive minors.

\bigskip

To prove $(i)\Rightarrow (ii)$, assume $(i)$.
Let $\mathcal{T}={\bf CM}(\mathcal{T}_m)$. By $(b)$ in Theorem 3, we have that ${\bf LO}(\mathcal{T})$ 
contains the binary equality constraint,
 the empty constraint, and it is closed under formation of conjunctive minors. Since
${\bf LO}(\mathcal{T})$ is
locally closed it follows from Theorem 2 that
 ${\bf LO}(\mathcal{T})$ is characterized by some set of 
$B$-valued functions of several variables on $A$, i.e.
${\bf LO}(\mathcal{T})= {\bf CSF}(\mathcal{K})$ for some set $\mathcal{K}$ 
of $B$-valued functions on $A$.
By Lemma 1, we have $\mathcal{T}_m=\mathcal{Q}_m\cap {\bf LO}(\mathcal{T})$.
 Thus $\mathcal{T}_m={\bf CSF}_m(\mathcal{K})$, i.e. $\mathcal{T}_m$ is characterized within $\mathcal{Q}_m$
 by some set of $B$-valued functions on $A$.
\end{proof}

Similarly to the Galois correspondences ${\bf FSC}_n-$${\bf CSF}$, the closure operators ${\bf FSC}\circ {\bf CSF}_m$ and 
${\bf CSF}_m\circ {\bf FSC}$ can be represented as compositions of ${\bf Lo}_m$ and ${\bf VS}$, and 
${\bf LO}$ and ${\bf CM}_m$, respectively. To establish such factorizations, we need the following:

\begin{theorem}
Consider arbitrary non-empty sets $A$ and $B$, and let $\mathcal{K}\subseteq \cup _{n\geq 1}B^{A^n}$ 
be a class of $B$-valued functions on $A$,
 and $\mathcal{T}_m$ be a set of $m$-ary $A$-to-$B$ relational constraints.
 The following hold:
\begin{itemize}
\item[(a)] If ${\bf VS}(\mathcal{K})=\mathcal{K}$,
 then ${\bf VS}({\bf Lo}_m(\mathcal{K}))={\bf Lo}_m(\mathcal{K})$; 
\item[(b)] If ${\bf CM}_m(\mathcal{T}_m)=\mathcal{T}_m$, then
 ${\bf CM}_m({\bf LO}(\mathcal{T}_m))={\bf LO}(\mathcal{T}_m)$.
\end{itemize}
\end{theorem} 

\begin{proof}
The proof $(a)$ can be obtained by replacing, in the proof of $(a)$ of Theorem 3,
 ${\bf Lo}$ by ${\bf Lo}_m$, and 
 ``finite subset $F$" by ``finite subset $F$ of size at most $m$". 
 
To prove $(b)$, we make use of $(b)$ in Theorem 3.
Let $\mathcal{Q}_m$ be the set of all $A$-to-$B$ $m$-ary relational constraints.
By Lemma 1 we have that 
$\mathcal{Q}_m\cap {\bf LO}({\bf CM}(\mathcal{T}_m))={\bf LO}({\bf CM}_m(\mathcal{T}_m))$, and
 by $(b)$ in Theorem 3, it follows that 
\begin{displaymath}
\mathcal{Q}_m\cap {\bf CM}({\bf LO}({\bf CM}(\mathcal{T}_m)))=\mathcal{Q}_m\cap {\bf LO}({\bf CM}(\mathcal{T}_m))={\bf LO}({\bf CM}_m(\mathcal{T}_m)).  
\end{displaymath}
Since ${\bf CM}_m({\bf LO}({\bf CM}(\mathcal{T}_m)))=\mathcal{Q}_m\cap {\bf CM}({\bf LO}({\bf CM}(\mathcal{T}_m)))$,
we get 
\begin{displaymath}
{\bf CM}_m({\bf LO}({\bf CM}(\mathcal{T}_m)))={\bf LO}({\bf CM}_m(\mathcal{T}_m)).  
\end{displaymath}
Observe that ${\bf CM}_m({\bf LO}(\mathcal{T}_m))\subseteq {\bf CM}_m({\bf LO}({\bf CM}(\mathcal{T}_m)))$,
 and 
${\bf LO}({\bf CM}_m(\mathcal{T}_m))={\bf LO}(\mathcal{T}_m)$ 
because $\mathcal{T}_m={\bf CM}_m(\mathcal{T}_m)$. 
Thus
\begin{displaymath}
{\bf CM}_m({\bf LO}(\mathcal{T}_m))\subseteq {\bf CM}_m({\bf LO}({\bf CM}(\mathcal{T}_m)))={\bf LO}(\mathcal{T}_m).  
\end{displaymath}
Since ${\bf LO}(\mathcal{T}_m)\subseteq {\bf CM}_m({\bf LO}(\mathcal{T}_m))$, we conclude that  
${\bf CM}_m({\bf LO}(\mathcal{T}_m))={\bf LO}(\mathcal{T}_m)$. 
\end{proof} 

Property $(a)$ in the above Theorem, is analogous to $(ii)$ of Lemma 2.5 in [P\" o2].  
From Theorem 9, Theorem 10 and Theorem 11, we obtain the analogue of Theorem 4.

\begin{theorem}
Consider arbitrary non-empty sets $A$ and $B$, and let $m$ be a positive integer.
 For any class of functions $\mathcal{K}\subseteq \cup _{n\geq 1}B^{A^n}$ and 
any set $\mathcal{T}_m$ of $m$-ary $A$-to-$B$ relational constraints, the following hold:
\begin{itemize}
\item[(i)] ${\bf FSC}({\bf CSF}_m(\mathcal{K}))= {\bf Lo}_m({\bf VS}(\mathcal{K}))$, and
\item[(ii)] ${\bf CSF}_m({\bf FSC}(\mathcal{T}_m))={\bf LO}({\bf CM}_m(\mathcal{T}_m))$.
\end{itemize}
\end{theorem}

\subsection{Simultaneous restrictions to the arities of functions and constraints}

Let $ \mathcal{K}$ 
be a class of $B$-valued functions on $A$,
 and $\mathcal{T}$ be a set of $A$-to-$B$ relational constraints. 
It is not difficult to see that for any positive integers $n$ and $m$,
 $B^{A^n}\cap {\bf Lo}_m(\mathcal{K})={\bf Lo}_m(B^{A^n}\cap \mathcal{K})$,
and that $\mathcal{Q}_m\cap {\bf LO}_n (\mathcal{T})= {\bf LO}_n(\mathcal{Q}_m\cap \mathcal{T})$, 
where $\mathcal{Q}_m$ denotes the set of all $A$-to-$B$ $m$-ary constraints.
Using these facts, Theorems 5 and 9, and Theorems 6 and 10 can be combined as follows:

\begin{theorem}
Consider arbitrary non-empty sets $A$ and $B$, and let $n$ and $m$ be positive integers.
 For a class of $n$-ary functions $\mathcal{K}_n\subseteq B^{A^n}$ 
the following conditions are equivalent:
\begin{itemize}
\item[(i)] $\mathcal{K}_n$ is $m$-locally closed and it is closed under $n$-ary simple variable substitutions;
\item[(ii)] $\mathcal{K}_n$ is definable within $B^{A^n}$ by some set of $A$-to-$B$ $m$-ary constraints.
\end{itemize}
\end{theorem}

\begin{proof} $(ii)\Rightarrow (i):$ Suppose that $(ii)$ holds, i.e. 
$\mathcal{K}_n = {\bf FSC}_n(\mathcal{T}_m)$ for some set $\mathcal{T}_m$ of $m$-ary constraints.
Let $ \mathcal{K}= {\bf FSC}(\mathcal{T}_m)$. By Theorem 9, 
$\mathcal{K}$ is $m$-locally closed and it is closed under simple variable substitutions.
Since $\mathcal{K}_n= B^{A^n}\cap \mathcal{K}$, $\mathcal{K}_n$ 
is closed under $n$-ary simple variable substitutions, and using the fact that 
$B^{A^n}\cap {\bf Lo}_m(\mathcal{K})={\bf Lo}_m(B^{A^n}\cap \mathcal{K})$, we conclude that
$\mathcal{K}_n$ is $m$-locally closed. Thus $(i)$ holds.
 
\smallskip

$(i)\Rightarrow (ii):$ Suppose that $(ii)$ holds, and let
 $ \mathcal{K}={\bf Lo}_m({\bf VS}(\mathcal{K}_n))$. By Lemma 1, we have that 
$\mathcal{K}_n= B^{A^n}\cap \mathcal{K}$, and it follows from Theorem 9 that 
$\mathcal{K}$ is definable by some set of $A$-to-$B$ $m$-ary constraints, i.e.
$ \mathcal{K}= {\bf FSC}(\mathcal{T}_m)$ for some set $\mathcal{T}_m$ of $m$-ary constraints.
Hence, $\mathcal{K}_n= B^{A^n}\cap {\bf FSC}(\mathcal{T}_m)={\bf FSC}_n(\mathcal{T}_m)$, i.e.
$(ii)$ holds.
\end{proof}

\begin{theorem}
Consider arbitrary non-empty sets $A$ and $B$ and let $n$ and $m$ be positive integers.
Let $\mathcal{Q}_m$ be the set of all $A$-to-$B$ $m$-ary relational constraints, and let
$\mathcal{T}_m\subseteq \mathcal{Q}_m$. 
Then the following are equivalent:
\begin{itemize}
\item[(i)]$\mathcal{T}_m$ is $n$-locally closed, contains the $m$-ary equality constraint and 
$m$-ary empty constraint, and it is
closed under formation of $m$-ary conjunctive minors;
\item[(ii)]$\mathcal{T}_m$ is characterized within $\mathcal{Q}_m$ by some set 
of $n$-ary $B$-valued functions on $A$.
\end{itemize}
\end{theorem}

\begin{proof} The proof of Theorem 14 follows in complete analogy with the proof of 
Theorem 13, using Theorem 6 and the remarks preceding Theorem 13.
\end{proof}

Furthermore, combining Theorem 8 and Theorem 12 we get:

\begin{theorem}
Consider arbitrary non-empty sets $A$ and $B$, and let $n$ and $m$ be positive integers.
 For any class of $n$-ary $B$-valued functions on $A$ and 
any set $\mathcal{T}_m$ of $m$-ary $A$-to-$B$ relational constraints, the following hold:
\begin{itemize}
\item[(i)] ${\bf FSC}_n({\bf CSF}_m(\mathcal{K}_n))= {\bf Lo}_m({\bf VS}_n(\mathcal{K}_n))$, and
\item[(ii)] ${\bf CSF}_m({\bf FSC}_n(\mathcal{T}_m))={\bf LO}_n({\bf CM}_m(\mathcal{T}_m))$.
\end{itemize}
\end{theorem}

\begin{proof}
From the above observations, and Theorem 12 (i) and Theorem 8 (ii), 
we get respectively, for
\begin{itemize} 
\item[(i)] ${\bf FSC}_n({\bf CSF}_m(\mathcal{K}_n))= B^{A^n}\cap {\bf FSC}({\bf CSF}_m(\mathcal{K}_n))= 
B^{A^n}\cap {\bf Lo}_m({\bf VS}(\mathcal{K}_n)) = {\bf Lo}_m(B^{A^n}\cap {\bf VS}(\mathcal{K}_n)) =
{\bf Lo}_m({\bf VS}_n(\mathcal{K}_n))$, and for 
\item[(ii)] ${\bf CSF}_m({\bf FSC}_n(\mathcal{T}_m))=
\mathcal{Q}_m\cap {\bf CSF}({\bf FSC}_n(\mathcal{T}_m))=
\mathcal{Q}_m\cap {\bf LO}_n({\bf CM}(\mathcal{T}_m)) =
{\bf LO}_n(\mathcal{Q}_m\cap {\bf CM}(\mathcal{T}_m)) =
{\bf LO}_n({\bf CM}_m(\mathcal{T}_m))$. 
\end{itemize}
\end{proof}

\end{document}